\documentclass{amsart}
\usepackage{graphicx}
\usepackage{amsmath}
\usepackage{amssymb}
\usepackage{amsthm}
\usepackage{a4wide}

\theoremstyle{plain}
\newtheorem{thm}{Theorem}

\newtheorem{prop}[thm]{Proposition}

\theoremstyle{definition}

\newcommand{\R}{\ensuremath{\mathbb{R}}}
\newcommand{\Z}{\ensuremath{\mathbb{Z}}}
\newcommand{\Q}{\ensuremath{\mathbb{Q}}}
\newcommand{\N}{\ensuremath{\mathbb{N}}}
\newcommand{\T}{\ensuremath{{\mathcal T}}}

\renewcommand{\epsilon}{\varepsilon}

\begin{document}

\title{Substitution Tilings with Dense Tile Orientations and $n-$Fold Rotational Symmetry}
\author[bielefeld]{D.~Frettl{\"o}h}
\author[ateneo]{A.L.D.~Say-awen}
\author[ateneo]{M.L.A.N.~De Las Pe\~{n}as}
\address{Ateneo de Manila University, Loyola Heights, Quezon City, Philippines}
\address{Bielefeld University, Postfach 100131, 33501 Bielefeld, Germany}
\date{\today}

\begin{abstract}
It is shown that there are primitive substitution tilings with
dense tile orientations invariant under $n$-fold rotation
for $n \in \{2,3,4,5,6,8\}$. The proof for dense tile orientations uses a
general result about irrationality of angles in certain parallelograms.
\end{abstract}

\maketitle

\section{Introduction}

From the discovery of Penrose tilings in the 70s \cite{Pen} and 
of quasicrystals in the 80s \cite{Shecht84} evolved a theory of
aperiodic order. 
One main method to produce interesting patterns showing aperiodic order
is a tile substitution. For a more precise description see below.
The idea is illustrated in Figure \ref{fig:pinw-2it}: a tile substitution
is a rule of how to enlarge a given {\em prototile} (or a set of several 
prototiles) and dissect it into congruent copies of the prototiles. The 
rule can be iterated to fill larger and larger regions of the plane.
\begin{figure}[hb]
\includegraphics[width=.7\textwidth]{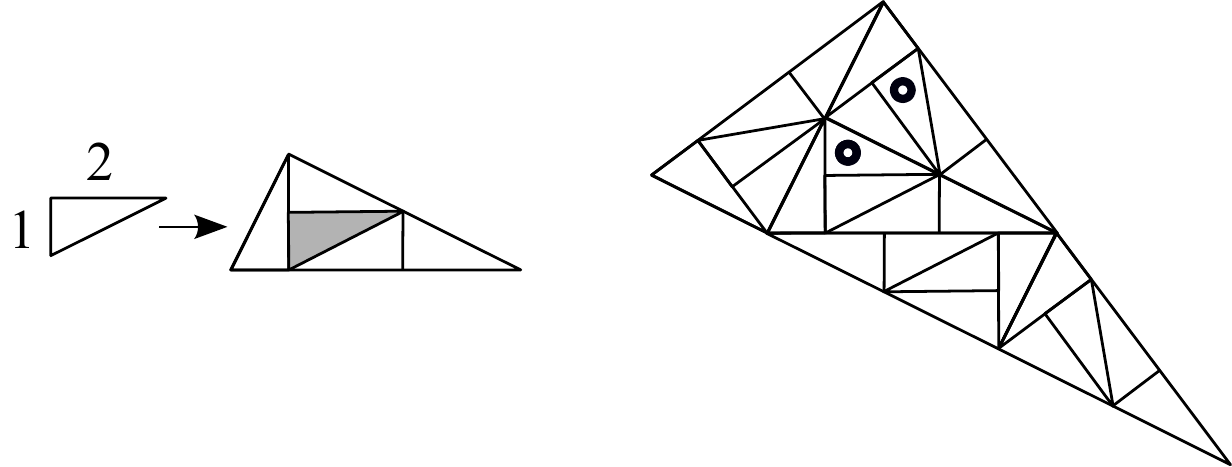}
\caption{Two iterations of the substitution for the pinwheel tiling.
\label{fig:pinw-2it}}
\end{figure}
Formally one considers a fixed point of the substitution: an infinite 
tiling $\T$ of the plane invariant under the substitution rule. This
fixed point $\T$ yields the {\em hull} of the tiling: the closure of 
the image $G \T$ of $\T$, where $G$ is a group acting on $\R^2$
(usually all translations in $\R^2$, or all rigid motions), and
closure is taken with respect to the local topology. For details see below,
for more details see for instance \cite{BaaGri13}.

Several mathematical fields interact in the theory of aperiodic order.
A lot of literature is dedicated to studying the topology of the hull 
of an aperiodic tiling. One way to do this is to compute its 
cohomology groups. For substitution tilings
this can be done by the methods introduced in \cite{AP98}. For the 
pinwheel tiling this was done in \cite{BDHS2010} and \cite{FWW2014}.
One problem for the pinwheel tiling is that the tiles occur in
infinitely many different orientations. More precisely: the
pinwheel tiling has {\em dense tile orientations} (DTO), i.e.
the orientations of the tiles are dense in the circle. For the
treatment of hulls of tilings with DTO in the context of dynamical systems 
see \cite{FreRi}. A further problem is that the hull of the pinwheel 
tiling contains six different tilings invariant under 2-fold rotation. 
These tilings correspond to cone singularities of the quotient 
of the hull by the circle. These give rise to a torsion part in the
second cohomology group $H^2$ of the hull. In particular, $m$ tilings
in the hull that are invariant under $n$-fold rotation contribute
a $\Z_n^{m-1}$ subgroup to $H^2$, where $\Z_n$ denotes the cyclic group of 
order $n$ \cite[Theorem 12]{BDHS2010}.
In view of this problem Jean Savinien \cite{Sav} asked in 2013 
for which values of $n$ primitive substitution tilings with DTO can 
be invariant under $n$-fold rotation. (For the definition of
primitivity see below.) This question motivated
this paper. Our main result is the following.
\begin{thm} \label{thm:main}
There are primitive substitution tilings with DTO that are invariant
under rotation by $\frac{2 \pi}{n}$ for $n \in \{2,3,4,5,6,8\}$. 
These tilings are not mirror symmetric, hence they occur in pairs
for each such $n$.
\end{thm}
Hence we have $m \ge 2$ in the discussion above, and the contributions
$\Z_n^{m-1}$ are not trivial. 
It is likely that the idea carries over to any $n \in \N$, but 
since the proof is constructive (and the constructions become
tedious for large $n$) we can deal only with the small cases here.
The cases $n \in \{3,4,6\}$ are considered in more detail in 
\cite{FreOySavSay}. The case $n=7$ is treated in \cite{april-phd}. The
case $n=2$ is known already to occur in pinwheel tilings.

This paper is organised as follows. Section \ref{sec:basics} contains 
some basic definitions and facts on substitution tilings. Readers 
familiar with this topic may skip this section. 
In order to show that all tilings have DTO we need a result
on the irrationality of certain angles. This is provided in
Theorem \ref{thm:irratangle} in Section \ref{sec:irrat}. 
The construction of the substitution rules is
given in Section \ref{sec:subst}. Theorem \ref{thm:main} is then
a consequence of Propositions \ref{prop:n=3+4}, \ref{prop:n=6},
\ref{prop:n=8}, and \ref{prop:n=5} in Section \ref{sec:subst}.

\section{Basics} \label{sec:basics}

For the purpose of this paper a {\em tile} is a nonempty 
compact set $T \subset \mathbb{R}^{2}$ which is the closure of its 
interior. A {\em tiling} of $\mathbb{R}^{2}$ is a collection of 
tiles $\mathcal{T}=\{T_{i} \, | \, i\in \mathbb{N}\}$ that is a
covering (i.e. $\bigcup_{i \in \mathbb{N}}T_{i}=\mathbb{R}^{2}$) as well
as a a packing (i.e. the intersection of the interiors of any two 
distinct tiles $T_{i}$ and $T_{j}$ is empty).  A finite subset of 
$\mathcal{T}$ is called a {\em patch} of $\mathcal{T}$. A tiling 
$\T$ has {\em finite local complexity} with respect to rigid motions 
(FLC for short) if for any $r>0$
there are only finitely many pairwise non-congruent patches in 
$\mathcal{T}$ fitting into a ball of radius $r$. (In many other 
contexts one may replace ``non-congruent'' by ``not translates of 
each other'', but here the first option is the appropriate one.)

 A tiling $\mathcal{T}$ is {\em nonperiodic}, if $\mathcal{T}+t=
\mathcal{T}$ $(t \in \mathbb{R}^{2})$ implies $t=0$. In addition, 
$\mathcal{T}$ is called {\em aperiodic} if each tiling in the hull 
of $\mathcal{T}$ is nonperiodic. The {\em hull} of the tiling 
$\mathcal{T}$ in $\mathbb{R}^{2}$ is the closure of the set 
$\{x \mathcal{T} \, | \, x \in G\}$ in the local topology.
Usually one takes $G=\R^2$ regarded as translations acting
on $\T$, or $G$ the group of all rigid motions in $\R^2$.
In our case it does not matter which one of the two we choose,
see \cite{FreRi}.
The local topology can be defined via a metric. In this metric
two tilings are $\epsilon$-close if they agree on a large ball
of radius $\frac{1}{\epsilon}$ around the origin, possibly after
a small motion (e.g. a translation by less than $\epsilon$).  
If $\mathcal{T}$ arises from a primitive substitution $\sigma$ one 
may as well speak of the {\em hull of $\sigma$}, since all tilings 
generated by $\sigma$ define the same hull. See for 
instance \cite{Sol97, PyF02, BaaGri13, MueRi, FreRi} for more details.

A {\em substitution rule} $\sigma$ is a simple method to generate 
nonperiodic tilings. A substitution rule consists of several 
{\em prototiles} $T_1, \ldots, T_m$, an {\em inflation factor} 
$\lambda>1$ and for each $i=1, \ldots, m$ a dissection of $\lambda T_i$
into congruent copies of some of the prototiles $T_1, \ldots, T_m$. 
The patch resulting from the dissection is denoted by $\sigma(T_i)$. 
A substitution $\sigma$ can be iterated on the resulting patch,
by inflating the patch by $\lambda$ and dissecting all tiles
according to $\sigma$. Hence it makes sense to write $\sigma^2(T_i)$
or $\sigma^k(T_i)$.   
A simple example is the substitution for the pinwheel tiling shown 
in Figure \ref{fig:pinw-2it}. This substitution uses just one prototile. 
The inflation factor is $\lambda = \sqrt{5}$. One may as well formulate
the pinwheel substitution for two prototiles: if we distinguish a
tile and its mirror image then the pinwheel substitution $\sigma_P$ has 
two prototiles $T_1$ and $T_2$ (where $T_2$ is the mirror image of $T_1$),
the substitution $\sigma_P(T_2)$ is the mirror image of $\sigma_P(T_1)$.

In certain instances we want to consider congruent tiles 
in $\mathcal{T}$ as different. This is achieved by markings
or colours. Two tiles are {\em equivalent} if they are congruent
and have the same marking or colour. See Subsection \ref{subsec:n=5}
below for an example where we consider congruent prototiles as 
different (e.g. $T_3^{(5)}, T_4^{(5)}, T_5^{(5)}$), with a different 
substitution for each prototile. 

Given a substitution $\sigma$ with prototiles $T_1, \ldots, T_m$
a patch of the form $\sigma(T_i)$ is called a {\em supertile}.
More generally, a  patch of the form $\sigma^k(T_i)$ is called 
a {\em $k$-th order supertile}. A substitution rule is called 
{\em primitive} if there is $k \in \N$ such that each $k$-th order 
supertile contains congruent copies of all prototiles.

Equivalently one may define primitivity of a substitution by
an associated matrix. The {\em substitution matrix} of 
a substitution $\sigma$ with prototiles $T_{1}, T_{2}, \cdots, T_{m}$
is $M_{\sigma}:=(a_{ij})_{1 \leq i,j \leq m}$, where $a_{ij}$ is the number 
of tiles equivalent to $T_{i}$ in $\sigma(T_{j})$, $i,j \in \{1,2, 
\ldots, m\}$. The substitution is primitive if and only if its
substitution matrix is primitive, which means that it has some power
containing positive entries only.
Primitivity is an important property for substitutions. One reason
is the following result, the Perron-Frobenius theorem.

\begin{thm}[\cite{Per09}] Let $M$ be a primitive non-negative square 
matrix. Then $M$ has a real eigenvalue $\lambda >0$ which is simple. 
Moreover, $\lambda > |\lambda'|$ for any eigenvalue $\lambda' \neq \lambda $. 
This eigenvalue $\lambda$ is called {\em Perron-Frobenius-eigenvalue} or 
PF-eigenvalue for short. Furthermore, the associated left and right 
eigenvectors of $\lambda$ can be chosen to have positive entries. 
Such eigenvectors are called the {\em left PF-eigenvector} and 
{\em right PF-eigenvector} of $M$.
\end{thm} 

Applied to a substitution tiling $\T$ the Perron Frobenius theorem 
has the following consequences, see for instance \cite{PyF02, BaaGri13}.

\begin{thm} \label{thm:ev}
Let $\sigma$ be a primitive substitution in $\mathbb{R}^{2}$ with 
inflation factor $\lambda$ and prototiles $T_{1}, T_{2}, \cdots, 
T_{m}$; let $M_{\sigma}$  be the substitution matrix of $\sigma$. Then 
the PF-eigenvalue of $M_{\sigma}$ is $\lambda^{2}$. The left PF-eigenvector 
contains the areas of the different prototiles, up to scaling. The 
normalised right PF-eigenvector $\boldsymbol{v}=(v_{1}, v_{2}, \cdots, 
v_{m})^{T}$ of $M_{\sigma}$ contains the relative frequencies of the 
prototiles of the tiling in the following sense: The entry $v_{i}$ is 
the relative frequency of $T_{i}$ in $\mathcal{T}$. 
\end{thm}

\section{An irrationality result} \label{sec:irrat}

An angle $\theta \in [0, 2 \pi[$ is called {\em irrational} if 
$\theta \notin \pi \Q$. The pinwheel tilings
have indeed DTO because of the fact that the second
order supertile contains two congruent tiles which are rotated
against each other by an irrational angle (see Figure \ref{fig:pinw-2it}
right, the two tiles are marked by dots). The angle here is 
$2 \arctan(1/2)$. It is known that $\arccos(\frac{1}{\sqrt{n}})
\notin \pi \Q$ for $n\ge 3$ odd \cite{AZbook}. Hence $2 \arctan(\frac{1}{2})
= 2(\frac{\pi}{2} - \arccos(\frac{1}{\sqrt{5}}))$ is irrational.
By induction the entire tiling contains tiles that are rotated against 
each other by $n \cdot 2 \arctan(\frac{1}{2}) \mod 2 \pi$ for all $n \in \N$. 
Since $2 \arctan(\frac{1}{2})$ is irrational these values are dense 
in a circle. More generally we have the following result:
\begin{thm}[{\cite[Proposition 3.4]{Fre08}}] \label{thm:dto-crit}
Let $\sigma$ be a primitive substitution in $\mathbb{R}^2$ with prototiles 
$T_{1}, T_{2}, \ldots, T_{m}$. Any substitution tiling in the hull of 
$\sigma$ has DTO if and only if there are $k,i$ such that 
$\sigma^{k}(T_{i})$ contains two equivalent tiles $T$ and $T'$ that are 
rotated against each other by some irrational angle.
\end{thm}
Hence the desired substitutions need to involve some irrational angles.
The following result provides such angles. The authors believe
that this result must be known already, but we are not aware of any 
reference. 
\begin{thm} \label{thm:irratangle}
Let $P$ be a parallelogram with edge lengths 1 and 2 and interior 
angles $\frac{2\pi}{n}$ and $\frac{(n-2)\pi}{n}$, $n \ge 3$. 
Then the angles between the longer diagonal of $P$ and the edges of $P$ 
are irrational. 
\end{thm}
\begin{figure}[h]
\centering
\includegraphics[width=.95\textwidth]{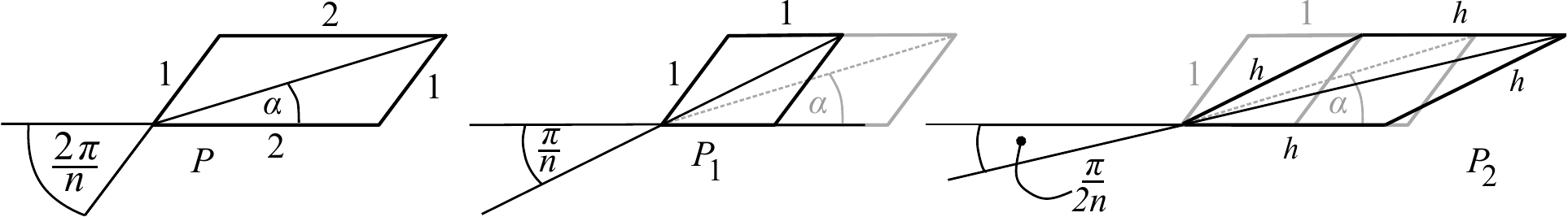}
\caption{The angle $\alpha$ (left); $\alpha$ is smaller than 
$\frac{\pi}{n}$ (middle); $\alpha$ is bigger than $\frac{\pi}{2n}$. 
\label{fig:irrat-winkel}}
\end{figure}
\begin{proof}
Embed $P$ in the complex plane such that the lower left corner of $P$ 
coincides with the origin and the lower base lies along the real axis 
as shown in Figure \ref{fig:irrat-winkel}. So the upper left corner 
of $P$ coincides with the point $\xi_{n}:=e^{2\pi i/n}$ and the upper 
right corner coincides with the point $z=\xi_{n}+2$. 
Let $\alpha$ be the angle between the long diagonal of $P$ and
the $x$-axis, see Figure \ref{fig:irrat-winkel} left. We will show that
$\alpha$ is irrational.

Suppose $\alpha$ is rational. Then there is $m \ge 1$
such that $z^m \in \mathbb{R}$, which yields 
\[ \left(\frac{z}{|z|}\right)^m= \pm 1 \Rightarrow 
\left( \frac{z}{|z|} \right)^{2m}=1 \; \Leftrightarrow  \;
\left( \frac{z^2}{z \overline{z}} \right)^m=1 \; \Leftrightarrow  \; \left( 
\frac{z}{\overline{z}} \right)^m=1, \]
hence $\frac{z}{\overline{z}}$ is some $m$-th root of unity. 

Because $z=\xi_{n}+2 \in \mathbb{Q}(\xi_{n})$ we have 
$\frac{z}{\overline{z}} \in \mathbb{Q}(\xi_{n})$. It is known (compare 
\cite[Exercise 2.3]{Wash97}) that roots of unity in $\mathbb{Q}(\xi_{n})$ 
are of the form $\pm \xi_{n}^{k}$, $0 \le k \le n-1$. Hence 
$\frac{z}{\overline{z}} = \pm \xi_{n}^{k}$ for some 
$k \in \mathbb{N}$. Then $2 \alpha = \arg\left(\frac{z}{\overline{z}}\right)=
\arg\left(\pm \xi_{n}^{k}\right)$, and so $\alpha=\frac{k\pi}{n}$
if $n$ is even, and $\alpha=\frac{k \pi}{2n}$ if $n$ is odd.

For $n$ even consider a second parallelogram $P_1$ with vertices 
$0,1,1+\xi_n,\xi_n$ (compare Figure \ref{fig:irrat-winkel} middle).
Since $\frac{\pi}{n}$ is the angle between 1 and the diagonal 
of $P_1$ we get $0<\alpha<\frac{\pi}{n}$, yielding a contradiction 
for $n$ even. 

For $n$ odd consider a third parallelogram $P_2$ with vertices $0,1+\xi_n,
1+\xi_n+h,h$, where $h$ is the length of the long diagonal of $P_1$
(compare Figure \ref{fig:irrat-winkel} right).  The angle between the
long diagonal of $P_2$ and the real axis is $\frac{\pi}{2n}$.
Since $h>1$ we obtain $\alpha>\frac{\pi}{2n}$. Together with the
reasoning above we get $\frac{\pi}{n}>\alpha>\frac{\pi}{2n}$, 
yielding a contradiction for $n$ odd. Therefore $\alpha$ is irrational.
\end{proof}

\section{Construction of the substitution tilings}
\label{sec:subst}

The general idea for the substitutions is to choose one prototile as
a bisected parallelogram from Theorem \ref{thm:irratangle}. To
be more precise, for $n \ge 3$ odd the prototile $T^{(n)}_2$ is the 
triangle with interior angle $\frac{n-1}{n} \pi$ where the two edges 
forming this angle have length one and two, respectively. 
For $n \ge 4$ even the prototile $T^{(n)}_2$ is the 
triangle with interior angle $\frac{n-2}{n} \pi$ where the two edges 
forming this angle have length one and two, respectively. 
By Theorem \ref{thm:irratangle} the other two angles of this triangle 
are irrational for any $n \ge 3$. A short computation yields the
length $\lambda_n$ of the longest edge as follows:
\begin{equation*}
\lambda_{n}=\left\{ 
\begin{array}{rc}
\sqrt{5+4\cos(\frac{\vphantom{2} \pi}{n})} & \mbox{if $n$ is odd}\\
\sqrt{5+4\cos(\frac{2\pi}{n})} & \mbox{if $n$ is even.}
\end{array}\right.
\end{equation*}
\begin{figure}
\includegraphics[width=.98\textwidth]{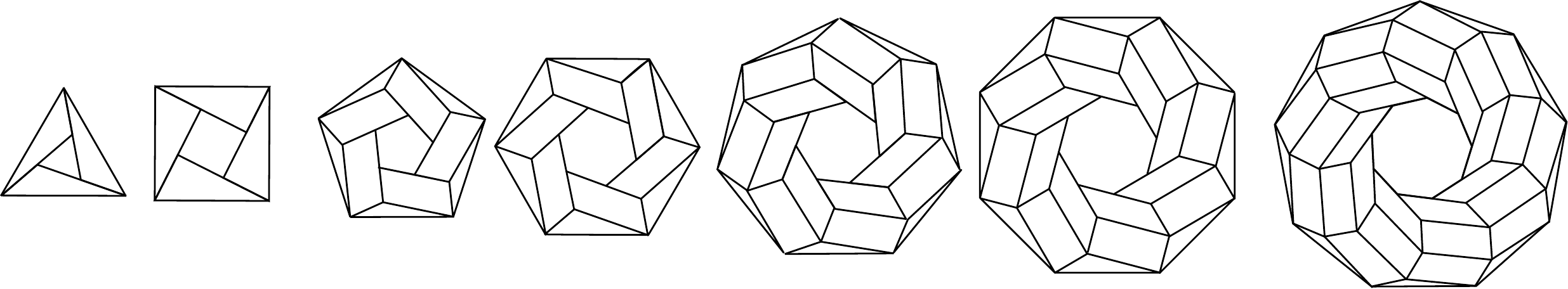}
\caption{Dissection of regular $n$-gons into a small regular $n$-gon 
$T^{(n)}_1$, $n$ triangles $T^{(n)}_2$ and possibly several parallelograms.  
\label{fig:n-ecke}}
\end{figure}
Let $\lambda_n$ be the inflation factor for the desired substitutions 
$\sigma_n$ in the sequel. A regular $n$-gon of side length $\lambda_n$ 
can be dissected into copies of $T^{(n)}_2$ (along its edges), one 
regular $n$-gon with unit edge length
(in its centre), and, if $n \ge 5$, into several parallelograms. 
This dissection is illustrated in Figure \ref{fig:n-ecke} for the 
cases $3 \le n \le 9$. 

In order to construct the desired substitution tilings with DTO being
invariant under $n$-fold rotation one chooses a first prototile 
$T^{(n)}_1$ to be a regular $n$-gon of unit edge length. 
The substitution of $T^{(n)}_1$ is then given by the dissection 
in Figure \ref{fig:n-ecke}. Therefore the inflation
factor equals $\lambda_n$. If one can find a substitution 
for all further prototiles arising in this dissection these substitutions 
are good candidates for DTO tilings since by Theorem \ref{thm:irratangle}
the central $n$-gon of edge length 1 is rotated against the big $n$-gon 
by an irrational angle. Furthermore---given a 
substitution exists for some $n$---the dissection of $\lambda_n T^{(n)}_1$ 
already provides a tiling invariant under $n$-fold rotation (given a
substitution for all tiles exists at all) since it may serve as a 
seed for a fixed point of $\sigma_n$ with $T^{(n)}_1$ in the centre.
(To be precise, one needs to define $\sigma_n$ including a rotational
part in order to take care of the different orientations of the
large and the small regular $n$-gons.)

\subsection{The 3-fold and 4-fold tilings}

The substitutions for $n \in \{3,4\}$ need only two prototiles. Two 
possible substitutions $\sigma_3$ and $\sigma_4$ are shown in Figure 
\ref{fig:subst-3-4}.
\begin{figure}
\includegraphics[width=90mm]{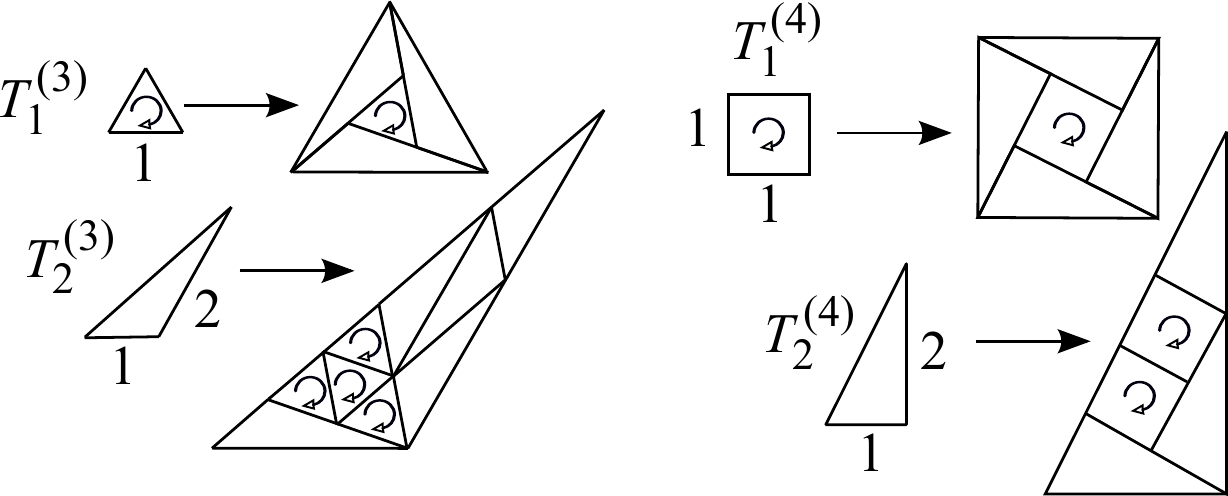}
\caption{The substitutions $\sigma_4$ (right) and $\sigma_3$
(left). For tiles with non-trivial symmetry the arrows indicate
the chirality of the tiles. By choice, all symmetric tiles in the images
are right-handed copies. \label{fig:subst-3-4}}
\end{figure}
\begin{prop} \label{prop:n=3+4}
For $n \in \{3,4\}$ holds: 
The substitution $\sigma_n$ is a primitive substitution with DTO. 
The hull of $\sigma_n$ contains two aperiodic tilings invariant under 
$n$-fold rotation. Any tiling in the hull of $\sigma_n$ is 
FLC with respect to rigid motions.
\end{prop}
\begin{proof}
Obviously the substitutions are primitive, the substitution matrices
are $M_{\sigma_3} = \big( \begin{smallmatrix} 1 & 4\\ 3 & 5 \end{smallmatrix} 
\big)$ and $M_{\sigma_4} = \big( \begin{smallmatrix} 1 & 2\\ 4 & 3 
\end{smallmatrix} \big)$.

Theorems \ref{thm:irratangle} and \ref{thm:dto-crit} imply that the 
tilings have DTO as follows. Let $\alpha$ denote the smallest interior
angle of $T_2^{(3)}$. Figure \ref{fig:3fd-3it} shows the situation
for $\sigma_3$: the grey shaded tile on the boundary of 
$\sigma_3^3(T_1^{(3)})$ is rotated by $2 \alpha$ against the grey shaded 
tile in the centre. Since $2 \alpha$ is irrational by Theorem 
\ref{thm:irratangle}, DTO of the tiling in the hull of $\sigma_3$
follows by Theorem \ref{thm:dto-crit}. This phenomenon appears
in all substitutions considered here and in the sequel: 
copies of $T_2^{(n)}$ are lined up
along the boundary of $\sigma_n(T_1^{(n)})$. Mirror images of $T_2^{(n)}$ 
are lined up along the boundary of $\sigma_n^2(T_1^{(n)})$ (since they
are mirror images Theorem
\ref{thm:dto-crit} does not apply here already), and rotated copies of
$T_2^{(n)}$ are lined up along the boundary of $\sigma_n^3(T_1^{(n)})$.
The boundaries of $\sigma_n(T_1^{(n)})$ and $\sigma_n^3(T_1^{(n)})$ 
are rotated against each other by $2 \alpha$, thus the 
triangles $T_2^{(n)}$ are rotated against each other by $2 \alpha$.

In order to show that the tilings are aperiodic it suffices to show
that the substitution $\sigma_n$ has a unique inverse on the hull
\cite[Theorem 10.1.1]{GruShe86}, see also \cite{Sol98,BaaGri13}. In the 
cases $n \in \{3,4\}$ this is particularly simple: For $n=3$ note that 
each isolated regular triangle $T^{(3)}_1$ is contained in a supertile
$\sigma_3(T_1^{(3)})$, hence the supertiles $\sigma_3(T_1^{(3)})$ can all be
identified uniquely. The remaining part of the tiling consists of supertiles 
$\sigma_3(T_2^{(3)})$, and the patches of four connected $T_1^{(3)}$ determine
the exact location and orientation of these supertiles. A similar
reasoning works for $n=4$. 

Let $R_\alpha$ denote the rotation about the origin through 
$\alpha$. Let $T_1^{(n)}$
be centred in the origin. Then $R_{\alpha} \sigma_n(T_1^{(n)})$ contains 
$T_1^{(n)}$ in its interior. Consequently, $(R_{\alpha} \sigma_n)^k(T_1^{(n)})$ 
contains $(R_{\alpha} \sigma_n)^{k-1}(T_1^{(n)})$ in its interior.
(Figure \ref{fig:3fd-3it} shows  $(R_{\alpha} \sigma_3)^k(T_1^{(3)})$ 
for $k=0,1,2,3$.) Hence $\big( 
(R_{\alpha} \sigma_n)^k(T_1^{(n)}) \big)_{k \in \N}$ is a nested sequence 
that converges in the local topology. Note that the local topology is 
usually defined for tilings, not for finite patches. Hence here we need
to extend the usual definition to patches, too, which is straightforward.
(Alternatively, one may extend the finite patches to tilings by
adding tiles. This does not change anything since we are only interested
in the central patches.) The limit is a tiling $\T$ that is fixed under 
$R_{\alpha} \sigma_n$. Since the tilings in the hull have DTO 
the hull is invariant under rotations. Thus the patches
$(R_{\alpha} \sigma_n)^k_n(T_1^{(n)})$ are legal in the sense that they 
are contained in tilings in the hull. Hence $\T$ is indeed contained 
in the hull. Since mirror images of all tiles occur also in all
tilings in the hull, the mirror image of $\T$ is also contained in 
the hull, yielding a second tiling invariant under $n$-fold rotation.
\begin{figure}[t]
\includegraphics[width=.9\textwidth]{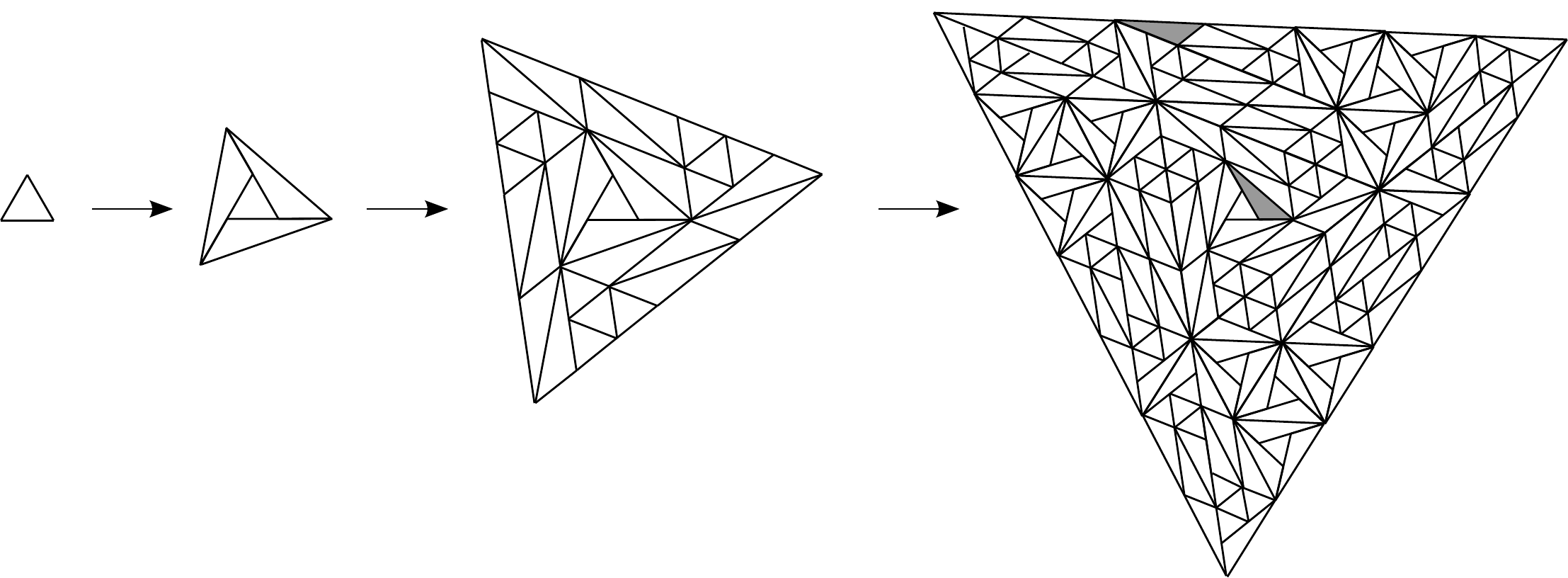}
\caption{Three iterations of $R_{\alpha} \sigma_3$ on $T_1^{(3)}$. The third 
order supertile $\sigma^3_3(T_1^{(3)})$ contains two copies of $T_2^{(3)}$ that
are rotated against each other by an irrational angle. \label{fig:3fd-3it}}
\end{figure}

We sketch why the tilings have FLC with respect to rigid motions. 
The simplest way to see this 
is to introduce an additional (pseudo-)vertex at the midpoint of the 
edge of length 2 in $T_2^{(n)}$. Taking into account this pseudo-vertex
the tilings are vertex-to-vertex. A standard argument implies that
the tilings are FLC. (There are finitely many ways how two tiles can 
touch each other, hence there are only finitely many possible patches 
fitting into a ball of radius $r$. A complete proof of FLC would need 
a list of all possible ways that two tiles can touch each other, 
e.g.\ a list of all vertex stars. Such a list is contained in 
\cite{FreOySavSay} for $n=3,4$. More details appear in \cite{april-phd}.)
\end{proof}

\subsection{The 6-fold tiling}

For $n=6$ we get the inflation factor $\lambda_{6}=\sqrt{5+4
\cos(\frac{2\pi}{6})}=\sqrt{7}$. The substitution $\sigma_6$ is shown 
in Figure \ref{fig:subst-6}. We need to introduce an additional tile
$T_4^{(6)}$ in order to ensure primitivity: $\lambda_6 T_3^{(6)}$ can
be dissected into congruent copies of $T_2^{(6)}$ and $T_3^{(6)}$,
but then $T_1^{(6)}$ would not occur in any of the supertiles
of $T_2^{(6)}$ and $T_3^{(6)}$, hence the substitution would not be 
primitive.

\begin{figure}[h]
\includegraphics[width=.9\textwidth]{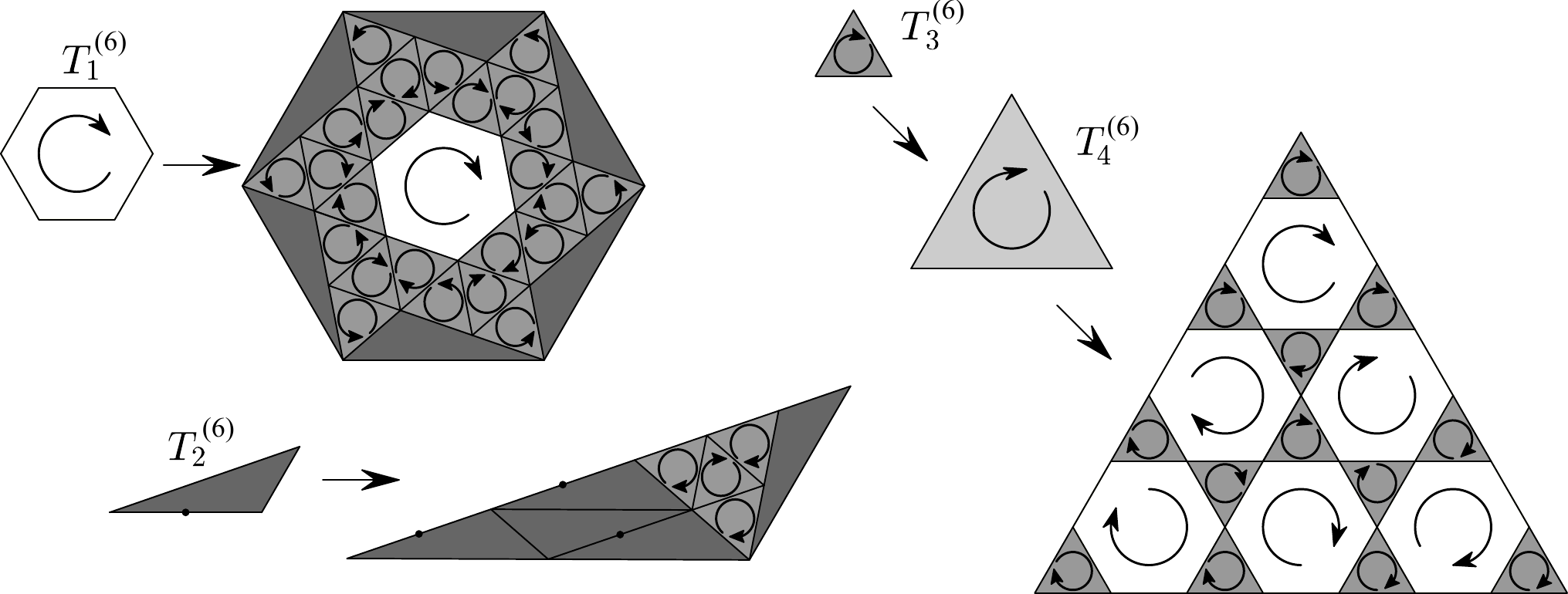}
\caption{The substitution $\sigma_{6}$. \label{fig:subst-6}}
\end{figure}

The substitution $\sigma_{6}$ has the substitution matrix
\[M_{\sigma_{6}}=\begin{pmatrix}
 1&0&0&6 \ \\       
 6&5&0&0 \ \\ 
 24&4&0&13 \ \\ 
 0&0&1&0 \ \\ 
\end{pmatrix},\]
which is primitive because $M_{\sigma_{6}}^{k}$ contains only 
positive entries for all $k \geq 3$. The corresponding PF-eigenvalue 
of $M_{\sigma_{6}}$ is $\lambda_{6}^{2}={7}$ with left PF-eigenvector 
$(6,2,1,7)$ and normalised right PF-eigenvector   
$(\frac{1}{12},\frac{1}{4},\frac{7}{12},\frac{1}{12})^T$.
By Theorem \ref{thm:ev} the left PF-eigenvector contains 
the areas of the tiles up to scaling, and the normalised right 
PF-eigenvector contains the relative frequencies of the tiles.
The latter can serve as a starting point for computing the frequency
module of the tilings. See \cite{FreOySavSay} for the computation of
the  frequency module of $\sigma_4$ by these means.

\begin{prop} \label{prop:n=6}
The substitution $\sigma_6$ is a primitive substitution with DTO. 
The hull contains two aperiodic tilings invariant under 6-fold 
rotation. Any tiling in the hull of $\sigma_6$  is FLC 
with respect to rigid motions.
\end{prop}
\begin{proof}
The proofs of DTO, FLC and the existence of a tiling invariant under 6-fold rotation
are very much along the lines in the proof of Proposition \ref{prop:n=3+4}. 
Aperiodicity of the tilings can be proven similarly by identifying the
first order supertiles uniquely: the tile $T_4^{(6)}$ occurs only as the 
supertile $\sigma_6(T_3^{(6)})$. A patch of six connected hexagons $T_1^{(6)}$ 
occurs only in the supertile $\sigma_6(T_4^{(6)})$. All remaining
hexagons $T_1^{(6)}$ determine the supertiles $\sigma(T_1^{(6)})$. All
remaining supertiles are $\sigma_6(T_2^{(6)})$. For more thorough 
proofs see \cite{april-phd}.
\end{proof}

\subsection{The 8-fold tiling}

For $n=8$ we get the inflation factor $\lambda_{8}=\sqrt{5+4\cos(\frac{2\pi}{8})}
=\sqrt{5+2 \sqrt{2}}$. The substitution rule is shown in Figure \ref{fig:subst-8}.
Since $\lambda_8 T_3^{(8)}$ and $\lambda_8 T_4^{(8)}$ cannot be dissected 
into copies of the prototiles
$T_1^{(8)}, T_2^{(8)}, T_3^{(8)}, T_4^{(8)}$ we need to introduce intermediate tiles
$T_5^{(8)} := \lambda_8 T_3^{(8)}$ and $T_6^{(8)} := \lambda_8 T_4^{(8)}$ in order to
define a substitution rule. Note that we need to define an orientation
on the tiles in order to distinguish a tile from its mirror image. This
is not indicated in the figure. There are several possibilities to do so.
One possibility is letting all tiles in figure have the same
orientation. The only point where this really matters is that the
tile $T_1^{(8)}$ in $\sigma_8(T_1^{(8)})$ has the same orientation as 
the prototile $T_1^{(8)}$ in order to ensure DTO.

\begin{figure}[h]
\includegraphics[width=.95\textwidth]{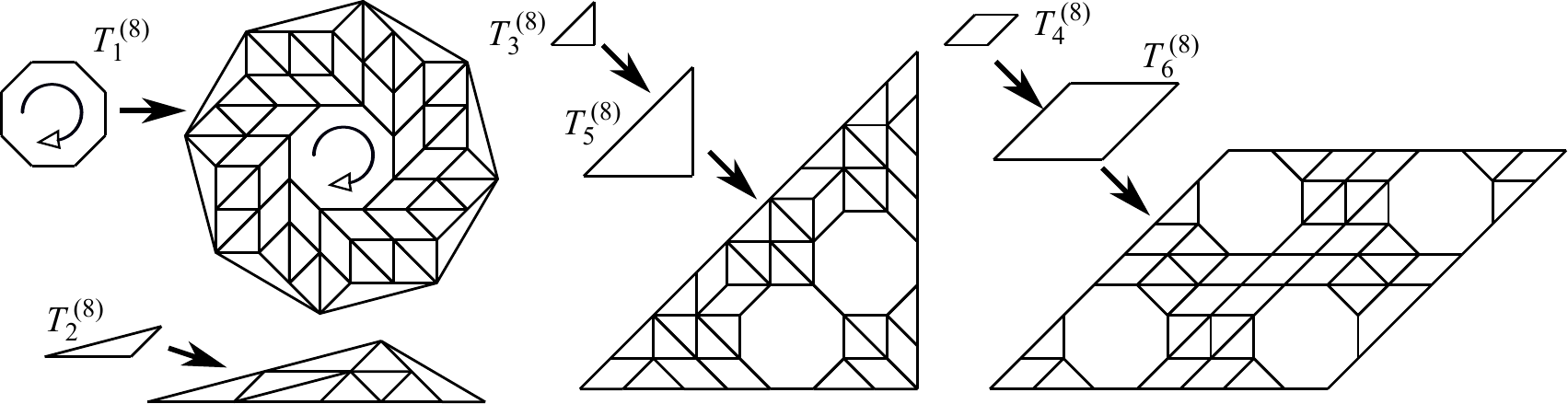}
\caption{The substitution $\sigma_{8}$. Orientations and chiralities 
of symmetric tiles are arbitrary if not shown in the image. 
\label{fig:subst-8}}
\end{figure}

\begin{prop} \label{prop:n=8}
The substitution $\sigma_8$ is a primitive substitution with DTO. 
The hull contains two aperiodic tilings invariant under 8-fold rotation.
\end{prop}
\begin{proof}
Again the proofs of DTO and the existence of a tiling invariant 
under 8-fold rotation
are very much along the lines in the proof of Proposition \ref{prop:n=3+4}.
Aperiodicity of the tilings 
can be proven similarly as above by identifying the
first order supertiles uniquely. More details appear in \cite{april-phd}. 
The primitivity can be checked via the substitution matrix 
\[ M_{\sigma_8} = \begin{pmatrix} 
1 & 0 & 0 & 0 & 2 & 4\\
8 & 5 & 0 & 0 & 0 & 0\\
32& 4 & 0 & 0 &25 &24\\
16& 0 & 0 & 0 &12 &17\\
0 & 0 & 1 & 0 & 0 & 0\\
0 & 0 & 0 & 1 & 0 & 0\\
\end{pmatrix}. \]
$M_{\sigma_8}^4$ has only positive entries. 
\end{proof}

We suppose that the tilings in the hull of $\sigma_8$ have also FLC. 
But since the tilings---with or without pseudo-vertices---are not 
vertex-to-vertex (this can be seen in $\sigma_8^3(T^{(8)}_1)$ for 
instance) a rigorous proof will be
rather lengthy. For details we refer to \cite{april-phd}.

\subsection{The 5-fold tiling} \label{subsec:n=5}

It is possible to define the desired substitution for $n=5$ using 
just six prototiles.
However the tilings may not have FLC. In order to ensure FLC we define 
a substitution $\sigma_{5}$ with 12 prototiles. The inflation factor is 
$\lambda_{5}=\sqrt{6+\sqrt{5}}$. The substitution rule is shown in Figure 
\ref{fig:subst-5}.

\begin{figure}[h]
\includegraphics[width=.9\textwidth]{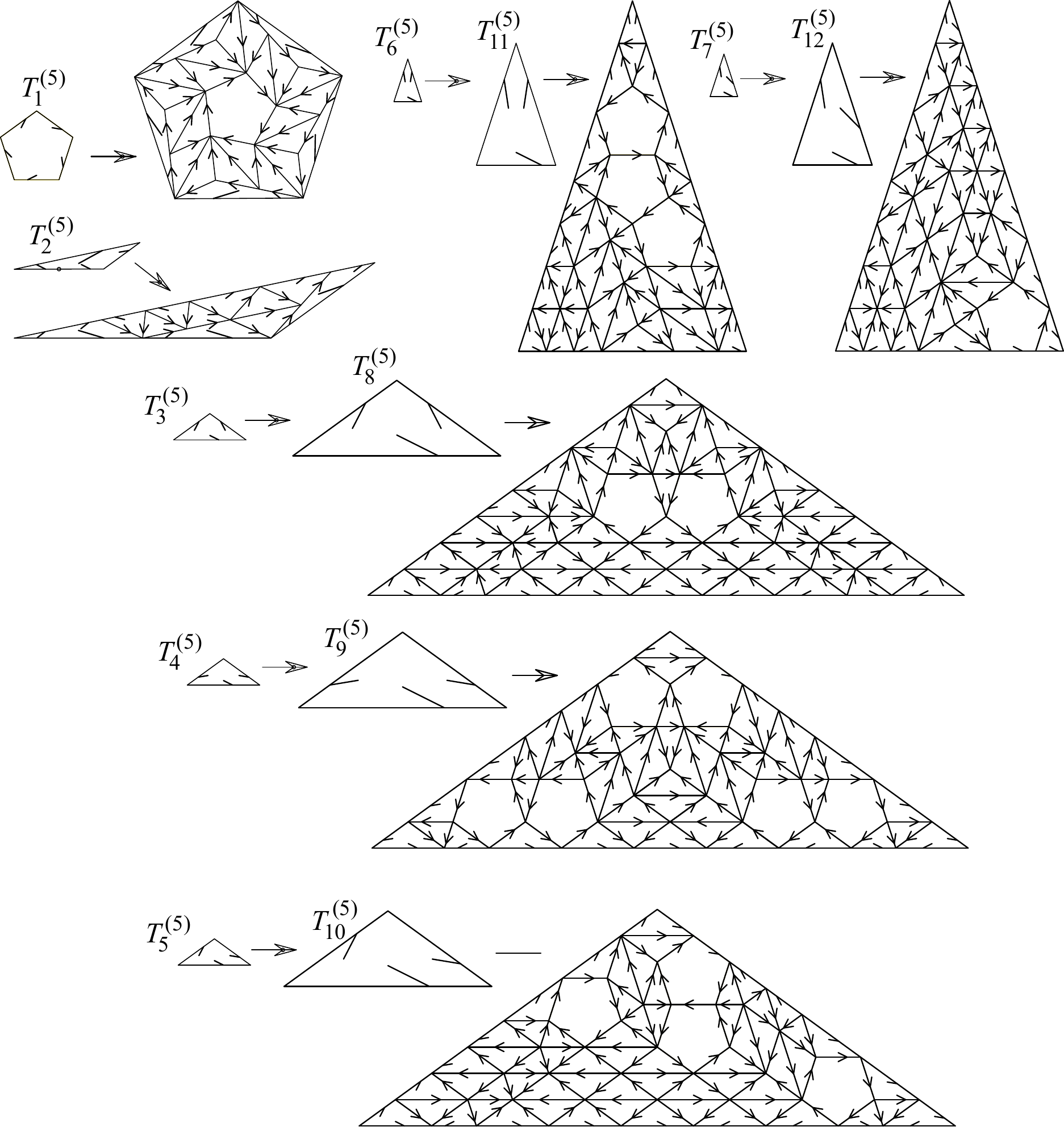} 
\caption{The substitution $\sigma_{5}$. Half arrows indicate the
orientation of edges and tiles. All coincident edges have the 
same orientation. 
\label{fig:subst-5}}
\end{figure}

\begin{prop} \label{prop:n=5}
The substitution $\sigma_5$ is a primitive substitution with DTO. 
The hull contains two aperiodic tilings invariant under 5-fold rotation.
Any tiling in the hull of $\sigma_5$ is FLC with respect to rigid motions.
\end{prop}
\begin{proof}
As before the proofs of DTO and the existence of a tiling invariant under 
5-fold rotation are along the lines of the proof of Proposition 
\ref{prop:n=3+4}. FLC follows from the fact that coincident edges
have the same orientation and that edges of the same length are dissected
in the same manner under $\sigma_5$.  
Aperiodicity of the tilings can be proven similarly by identifying the
first order supertiles uniquely. More details appear in \cite{april-phd}. 
Primitivity of $\sigma_5$ can be checked by considering the substitution 
matrix $M_{\sigma_5}$ below. Since $M_{\sigma_5}^5$ contains only positive 
entries the substitution $\sigma_5$ is primitive. 
\end{proof}

The substitution matrix for $\sigma_5$ looks as follows.
{\small 
\[M_{\sigma_{5}}=\begin{pmatrix}
 1&0&0&0&0&0&0&2&6&4&3&1 \ \\       
 5&3&0&0&0&0&0&0&0&0&0&0 \ \\
 0&0&0&0&0&0&0&18&9&10&5&11 \ \\
 5&0&0&0&0&0&0&21&8&6&1&2 \ \\
 15&4&0&0&0&0&0&8&18&25&9&8 \ \\
 0&2&0&0&0&0&0&12&6&7&9&11 \ \\
 0&2&0&0&0&0&0&10&12&13&17&17 \ \\
 0&0&1&0&0&0&0&0&0&0&0&0 \ \\
 0&0&0&1&0&0&0&0&0&0&0&0 \ \\
 0&0&0&0&1&0&0&0&0&0&0&0 \ \\
 0&0&0&0&0&1&0&0&0&0&0&0 \ \\
 0&0&0&0&0&0&1&0&0&0&0&0 \ \\
\end{pmatrix}.\]}
The powers of $M_{\sigma_5}$, the 
PF-eigenvalue and the PF-eigenvectors have been computed with 
the computer algebra software (CAS) {\em Scientific Workplace 5.5} 
\cite{SW}. The PF-eigenvalue is $\sqrt{5}+6$ which equals $\lambda_5^2$, 
as it ought to, and its corresponding left PF-eigenvector and 
normalised right PF-eigenvector are respectively given by  
{\small
\begin{equation*}
\boldsymbol{w}=\begin{pmatrix}
\frac{15}{62}+\frac{13}{62}\sqrt{5}\\[1mm]
\frac{12}{31}-\frac{2}{31}\sqrt{5}\\[1mm]
\frac{1}{62}+\frac{5}{62}\sqrt{5}\\[1mm]
\frac{1}{62}+\frac{5}{62}\sqrt{5}\\[1mm]
\frac{1}{62}+\frac{5}{62}\sqrt{5}\\[1mm]
\frac{1}{6+\sqrt{5}}\\[1mm]
\frac{1}{6+\sqrt{5}}\\[1mm]
\frac{1+\sqrt{5}}{2}\\[1mm]
\frac{1+\sqrt{5}}{2}\\[1mm]
\frac{1+\sqrt{5}}{2}\\[1mm]
1\\[1mm]
1\\
\end{pmatrix}^{T}\ \ \ \ 
\mbox{and}\ \ \ \ \ \ 
\boldsymbol{v}=\begin{pmatrix}
\frac{2640247257}{109 180718845}\sqrt{5}-\frac{233289537}{21 836143769}\ \\[1mm]
\frac{2271797364}{21 836143769}\sqrt{5}-\frac{4175144835}{21 836143769}\ \\[1mm]
\frac{-69 639647193}{349 378300304}\sqrt{5}+\frac{200 370426489}{349 378300304}\ \\[1mm]
\frac{-281 379644707}{1746 891501520}\sqrt{5}+\frac{164 659760407}{349 378300304}\ \\[1mm] 
\frac{16 679139843}{21 836143769}\sqrt{5}-\frac{31 075028997}{21 836143769}\ \\[1mm]
\frac{-208 504378761}{873 445750760}\sqrt{5}+\frac{113 488825221}{174 689150152}\ \\[1mm]
\frac{-214 190232831}{873 445750760}\sqrt{5}+\frac{126 605432787}{174 689150152}\ \\[1mm]
\frac{-19 942203537}{349 378300304}\sqrt{5}+\frac{50 013574029}{349 378300304}\ \\[1mm]
\frac{-81 018602267}{1746 891501520}\sqrt{5}+\frac{40 946393779}{349 378300304}\ \\[1mm]
\frac{4230640905}{21 836143769}\sqrt{5}-\frac{8704705587}{21 836143769}\ \\[1mm]
\frac{-58 660335441}{873 445750760}\sqrt{5}+\frac{28 691526777}{174 689150152}\ \\[1mm]
\frac{-61 876405191}{873 445750760}\sqrt{5}+\frac{31 413639663}{174 689150152}\ \\
\end{pmatrix}
\end{equation*}
}
Again the left PF-eigenvector $\boldsymbol{w}$ contains 
the areas of the tiles up to scaling, and the right PF-eigenvector 
$\boldsymbol{v}$ contains the relative frequencies of the tiles.

\section{Conclusion}

We were not able to define a general substitution rule $\sigma_{n}$ 
for all $n$, or at least for all even $n$. Anyway, there is
a general pattern for dissecting $\lambda_n T_2^{(n)}$ ($n \ne 4$) into 
five copies of $T_2^{(n)}$ and four further triangles, and for dissecting 
$\lambda_n T_1^{(n)}$ into one copy of $T_1^{(n)}$, $n$ copies of $T_2^{(n)}$ 
and several parallelograms. Hence it is likely that there are substitution 
tilings with DTO invariant under $n$-fold rotation for all $n \ge 3$. 
This might be also of interest with respect to a comment in \cite{Mal15}:
``...there is a lack of known examples of aperiodic planar tiling families 
with higher orders of rotational symmetry.''. That paper contains
``the first substitution tiling with elevenfold symmetry appearing in the
literature''. Our method might yield further primitive substitution tilings 
with 11-fold and also 12-fold rotational symmetry (though the number of 
prototiles might be huge). Nevertheless, after submission of this paper
we became aware of the work of Kari and Rissanen \cite{KaRi16} containing
substitution tilings with $2n$-fold symmetry for arbitrary $n$.

The proof of Theorem \ref{thm:irratangle} on irrational angles in cyclotomic 
parallelograms uses the fact that the considered irrational angle $\alpha$ 
is smaller than $\frac{\pi}{n}$. Hence the result generalises immediately
to the long diagonals of parallelograms with interior angle $\frac{\pi}{n}$
and with edge lengths $a \ne b \in \Q$ rather than 2 and 1. Using other 
arguments it might be possible to show the irrationality of other angles 
as well, e.g. the angle of the short diagonal.

For the sake of briefness we did not mention further implications of
our constructions in the context of dynamical properties of the hull.
Just to mention a few: the fact that a tiling has FLC ensures the
compactness of the hull of $\sigma_n$. Due to primitivity of $\sigma_n$
all tilings in the hull of $\sigma_n$ are repetitive, hence (1) all tilings 
have uniform patch frequencies, and (2) the hull is minimal. As a 
consequence we obtain: If we denote the hull of $\sigma_n$ by $X_{\sigma_n}$ 
then the dynamical systems $( X_{\sigma_n}, \R^2)$ and $( X_{\sigma_n}, E(2))$ 
(where $E(2)$ denotes the rigid motions in $\R^2$) are 
both uniquely ergodic. For more details on these concepts see
\cite{Sol97, BaaGri13, FreRi}.

\section*{Acknowledgements}
The authors express their gratitude to Franz G\"ahler, Lorenzo Sadun 
and Mike Whittaker for helpful discussions. M.~De  Las Pe\~nas 
acknowledges the support of the German Academic Exchange Service (DAAD) 
(research stay programme). A.~Say-awen gives thanks to the Philippine 
Department of Science and Technology (DOST) (DOST‐ASTHRDP scholarship).

\end{document}